\documentclass[a4paper,12pt,leqno]{amsart}
\usepackage{latexsym}
\usepackage{amsmath}
\usepackage{amssymb}
\usepackage{amsthm}
\usepackage{indentfirst}
\usepackage[english]{babel} 
\usepackage[T1]{fontenc}
\usepackage{breqn}
\usepackage{cite}    
\usepackage{mathrsfs}                   
\usepackage[colorlinks=true]{hyperref}

\theoremstyle{plain} 
\newtheorem{thm}{Theorem}[section] 
 
\newtheorem{lem}[thm]{Lemma} 
 
\newtheorem{rmk}[thm]{Remark}

\theoremstyle{definition} 
 
\newcommand{\eps}{\varepsilon}

\numberwithin{equation}{section}

\author[D.~Mastrostefano]{Daniele Mastrostefano}
\address{University of Warwick, Mathematics Institute, Zeeman Building, Coventry, CV4 7AL, UK}
\email{Daniele.Mastrostefano@warwick.ac.uk}

\keywords{Random multiplicative functions; law of iterated logarithm; Borel--Cantelli lemma; sums of independent random variables; low moments}
\subjclass[2010]{Primary: 11K65. Secondary: 11N64.}

\begin{document}
\title[Almost sure upper bound random multiplicative functions]{An almost sure upper bound for random multiplicative functions on integers with a large prime factor}\thanks{The author is funded by a Departmental Award and by an EPSRC Doctoral Training Partnership Award. The present work was carried out when the author was a third year PhD student at the University of Warwick.}

\begin{abstract}
Let $f$ be a Rademacher or a Steinhaus random multiplicative function. Let $\eps>0$ small. We prove that, as $x\rightarrow +\infty$, we almost surely have
$$\bigg|\sum_{\substack{n\leq x\\ P(n)>\sqrt{x}}}f(n)\bigg|\leq\sqrt{x}(\log\log x)^{1/4+\eps},$$
where $P(n)$ stands for the largest prime factor of $n$. This gives an indication of the almost sure size of the largest fluctuations of $f$.
\end{abstract}

\maketitle

\section{Introduction}
A fundamental and classical problem in Analytic Number Theory concerns demonstrating squareroot cancellation for the partial sums of the M\"obius function $\mu(n)$, defined as the multiplicative function supported on the squarefree numbers and attaining value $-1$ on the primes. More precisely, one ponders the validity of the following statement:
\begin{align*}
\sum_{n\leq x} \mu(n)\ll_\eps x^{1/2+\eps}
\end{align*}
for all $\eps>0$ and $x$ large with respect to $\eps$, which is equivalent to the Riemann hypothesis (see Soundararajan \cite{SOUND} for a refinement of such relation).

To investigate this problem, Wintner \cite{W}, in 1944, introduced the following model for $\mu(n)$:

\emph{a Rademacher random multiplicative function $f$ is a multiplicative function supported on the squarefree integers and defined on the prime numbers $p$ by letting the $f(p)$ be independent
random variables taking values $\pm 1$ with probability $1/2$ each.}

Clearly, $f$ and $\mu$ are both multiplicative, supported on the squarefree numbers and take values $\pm 1$. So, a Rademacher random multiplicative function represents a reasonable heuristic model for the M\"obius function (see the introduction to \cite{LTW} for more discussion about this).

In fact, Wintner himself \cite{W} was able to show that, for any fixed $\eps>0$, one almost surely has
\begin{align*}
&\sum_{n\leq x} f(n) = O(x^{1/2+\eps})\\
&\sum_{n\leq x} f(n) \neq O(x^{1/2-\eps}),
\end{align*}  
thus implying that the Riemann hypothesis is ``almost always'' (related to this probabilistic model) true.

These results have been later improved in an unpublished work by Erd\H{o}s \cite{E} and after by Hal\'{a}sz \cite{H}, culminating in the work of Basquin \cite{BA} and independently Lau, Tenenbaum and Wu \cite{LTW}, who gave, for any $\eps>0$, the following almost sure upper bound
\begin{align}
\label{LTWresult}
\bigg|\sum_{n\leq x}f(n)\bigg|\leq \sqrt{x}(\log \log x)^{2+\eps}\ \text{as $x\rightarrow +\infty$}.
\end{align}
On the opposite side, Harper \cite{H2}, improving on his own previous result \cite{H1}, showed that, for any function $V(x)$ tending to infinity with $x$, there almost surely exist arbitrarily large values of $x$ for which
\begin{align}
\label{Adamresult}
\bigg|\sum_{n\leq x} f(n)\bigg|\geq \frac{\sqrt{x}(\log\log x)^{1/4}}{V(x)}. 
\end{align}
This result also holds for Steinhaus random multiplicative functions $f$, where $\{f(p)\}_{p\ \text{prime}}$ is a sequence of independent Steinhaus random variables (i.e. distributed uniformly on the
unit circle $\{|z| = 1\}$) and the function $f$ is taken to be completely multiplicative. 

He achieved \eqref{Adamresult} by reducing the problem to showing a similar statement, but where the sum runs only over integers with the largest prime factor $>\sqrt{x}$ (actually, he worked with a slightly different condition, but his argument may be adapted to this case). At the same time, he localised the problem by considering such statement only for a collection of values of $x$ simultaneously. Then, he showed a  multivariate Gaussian approximation for such sums, conditional on the behaviour of $f$ at small primes. Finally, he controlled the interaction among these sums as $x$ varies, showing conditional almost independency, through a careful and delicate study of the size of their covariances.

In particular, as a consequence of the proof of \eqref{Adamresult}, we may infer that there almost surely exist arbitrarily large values of $x$ for which
\begin{align}
\label{Adamresult2}
\bigg|\sum_{\substack{n\leq x\\ P(n)>\sqrt{x}}} f(n)\bigg|\geq\sqrt{x}(\log\log x)^{1/4+o(1)}.
\end{align}
The bounds \eqref{LTWresult} and \eqref{Adamresult} together give the feeling of the existence of a \emph{Law of the Iterated Logarithm} for the partial sums of $f(n)$. 

For any sequence of \emph{independent} random variables $(\eps_n)_{n=1}^{+\infty}$ taking values $\pm 1$ with probability $1/2$ each, Khintchine's Law of
the Iterated Logarithm consists in the following almost sure statements:
$$\limsup_{x\rightarrow +\infty}\frac{\sum_{n\leq x}\eps_n}{\sqrt{2x\log\log x}}=1\ \ \ \text{and}\ \ \ \liminf_{x\rightarrow +\infty}\frac{\sum_{n\leq x}\eps_n}{\sqrt{2x\log\log x}}=-1.$$
See for instance Gut \cite[Ch. 8]{G} for an extensive account of this result. 
To compare, for large $x$, the sum $\sum_{n\leq x}\eps_n$ has a roughly Gaussian distribution with mean zero and variance $x$, by the Central Limit Theorem. So, it typically has size close to $\sqrt{x}$, from which its largest fluctuations are obtained by rescaling this size by a $\sqrt{\log\log x}$ factor, which describes the impact of the dependence amongst the sums $\sum_{n\leq x}\eps_n$ as $x$ varies.

Khintchine's theorem cannot be applied to study random multiplicative functions, because their values are clearly not all independent. Nevertheless, we might believe that a suitable version of the Law of
the Iterated Logarithm might hold for them, in the hope that their multiplicative structure does not completely disrupt statistical cancellations.
However, the exact size of the almost sure largest fluctuations of their partial sums is not yet clear. Following Harper \cite{H2}, in relation to Khintchine's law we might reason that, for a Rademacher or Steinhaus random multiplicative function $f$, it might be obtained by adjusting the typical size $\sqrt{x}/(\log\log x)^{1/4}$ of the partial sums of $f$, which now does no more coincide with their standard deviation $\sqrt{x}$ (see \cite[Corollary 2]{H3}), with the usual Law of the Iterated Logarithm ``correction factor'' $\sqrt{\log\log x}$, that in \eqref{Adamresult} is a result of the previously mentioned multivariate Gaussian approximation. So doing, we arrive at the following prediction:\footnote{The lower bound is already given by \eqref{Adamresult} and the validity of the upper bound is currently being investigated in a joint project with Adam J. Harper.}

\emph{we almost surely have 
\begin{align*}
\bigg|\sum_{n\leq x} f(n)\bigg|\leq \sqrt{x}(\log\log x)^{1/4+o(1)}\ \text{as $x\rightarrow +\infty$}
\end{align*}
and the opposite inequality almost surely holds on a subsequence of points $x$.}

Recall that $P(n)$ indicate the largest prime factor of $n$. The following theorem may be seen as a partial result in this direction.
\begin{thm}
\label{newthm}
Let $f$ be a Rademacher or a Steinhaus random multiplicative function. Let $\eps>0$ small. As $x\rightarrow +\infty$, we almost surely have
$$\bigg|\sum_{\substack{n\leq x\\ P(n)>\sqrt{x}}}f(n)\bigg|\leq\sqrt{x}(\log\log x)^{1/4+\eps}.$$
\end{thm}
\begin{rmk}
Considering \eqref{Adamresult2}, the bound in Theorem \ref{newthm} is close to be sharp.
\end{rmk}
\subsection{Sketch of the proof of Theorem \ref{newthm}} 
As usual when seeking to produce almost sure bounds for a sum of random variables, we will reduce our analysis to what happens on a sequence of ``test points'', with the property of being sparse, but not too much to yet easily control the increments of $f$ between two consecutive elements of such sequence. We will then collect together the information we gather from each single point, by means of the first Borel--Cantelli's lemma. For random multiplicative functions this is indeed the approach that was taken by Basquin \cite{BA} and Lau--Tenenbaum--Wu \cite{LTW} and others before, but the way we analyze the distribution of the partial sums of $f(n)$ on test points is a key difference with them. In fact, the study of their distribution is made available through the use of high moments inequalities. To avoid them blowing up, Basquin and Lau--Tenenbaum--Wu split the \emph{full} partial sums of $f(n)$ up into several pieces where the constraints were on the size of the largest prime factor in intervals. Here, we are in an easier setting, since we have to deal with just \emph{one} such intervals where moreover there is a \emph{unique} large prime factor, which leads to improving the efficiency of the high moments bounds. This allows us to saving a $\log\log x$ factor in the estimate of Theorem \ref{newthm} compared to \cite{BA} and \cite{LTW}. More specifically, we note that any positive integer $n\leq x$ with $P(n)>\sqrt{x}$ can be uniquely written as $n=pm$, where $\sqrt{x}<p\leq x$ is a prime and $m\leq x/p$ is a positive integer. Consequently, by multiplicativity, we deduce that
\begin{align}
\label{sumsindeprv}
\sum_{\substack{n\leq x\\ P(n)>\sqrt{x}}}f(n)=\sum_{\sqrt{x}<p\leq x} f(p)\sum_{m\leq x/p}f(m).
\end{align}
Conditional on the value of $f(q)$, for prime numbers $q\leq \sqrt{x}$, the above can be interpreted as a sum of many independent random variables, which conditional probability distribution possesses a conditional \emph{Gaussian} tail, thanks to Hoeffding's inequality. This is exactly how we gain the $\log\log x$ factor mentioned above, by replacing the use of several high moments bounds (one for each of the roughly $\log\log x$ sums related to the size of the largest prime factor, as in \cite{BA} and \cite{LTW}) with that of a single one. 

The use of Hoeffding's inequality constitutes a difference also in relation to Harper's lower bound result \eqref{Adamresult}, where instead, as explained before, there was needed to establish a conditional jointly Gaussian approximation for the partial sums of $f$ over integers with a large prime factor (which required in \cite{H2} a much greater effort).

Our second conditioning will be on the size of a certain smooth weighted version of the conditional variance $V(x)$ of the partial sums in \eqref{sumsindeprv}, which is equal to $\sum_{\sqrt{x}<p\leq x} |\sum_{m\leq x/p}f(m)|^2$. Arguing as in Harper \cite{H3}, we will recast it in terms of an $L^2$-integral of a truncated Euler product corresponding to $f$, which will give rise to a submartingale sequence.

If $(\Omega, \mathcal{F}, \mathbb{P})$ is a probability space with, additionally,
a sequence $\{\mathcal{F}_n\}_{n\geq 0}$ of increasing sub-$\sigma$-algebras of $\mathcal{F}$, which is called a \emph{filtration}, a \emph{submartingale} is a sequence of random variables $(X_n)_{n\geq 0}$ which satisfies, for any time $n$, the following properties:
\begin{align*}
&X_n\ \text{is}\ \mathcal{F}_n-\text{measurable} && \text{($X_n$ is adapted)}\\
&\mathbb{E}[|X_n|]<+\infty && (\text{$X_n$ is integrable})\\
&\mathbb{E}[X_{n+1}| \mathcal{F}_n]\geq X_n \ \text{almost surely} && (\text{$X_n$ is non-decreasing on average}).
\end{align*}
Rewriting a smooth version of the partial sums of $f$ in terms of a submartingale sequence is a common feature with \cite{BA} and \cite{LTW}. However, differently from them, our sequence involves the Euler product of $f$ and, most importantly, we will be able to input \emph{low moments} estimates for the partial sums of $f$ to better bound its size, which will allow us to gain a further $(\log\log x)^{1/4}$ factor compared to \cite{BA} and \cite{LTW}. More precisely, the gain will come from showing that with high probability $V(x)$ has uniformly size close to $x/\sqrt{\log\log x}$, which is what we can pointwise deduce from Harper's low moments estimates \cite{H3}.

To implement such results successfully, we will need to drastically increase the number of test points considered in contrast with \cite{BA} and \cite{LTW}. This will force us to introduce a suitable \emph{normalized} version of the aforementioned submartingale sequence, with the normalization given by the reciprocal of its expected value times a correction factor.
So doing, and by means of Doob's maximal inequality to control such sequence uniformly on test points, we will decrease of a $\log\log x$ factor, compared to \cite{BA} and \cite{LTW}, the gap between the actual size of $V(x)$ and its uniform value $x/\sqrt{\log\log x}$ (whereas in previous works, $V(x)$ was put in relation with its expected size $x$ and the precision loss was indeed roughly $\log\log x$). This will lead to a last gain of roughly $\sqrt{\log\log x}$ in Theorem \ref{newthm}, thus overall reducing the upper bound from $(\log\log x)^{2+\eps}$ to $(\log\log x)^{1/4+\eps}$, in our case.

To recap, unlike the approach taken in \cite{BA} and \cite{LTW}, we are going to introduce three key tools, which, compared to the result obtained in \cite{BA} and \cite{LTW}, permit us to save:
\begin{itemize}
\item a $\log\log x$ factor, by improving the use of the high moments inequalities to study the distribution of the partial sums of $f$, having a single large prime factor to take out;
\item a $(\log\log x)^{1/4}$ factor, by inputting low moments estimates for the full partial sums of $f$ into our argument;
\item a final $\sqrt{\log\log x}$ factor, by analysing the partial sums of $f$ over a larger sample of points and simultaneously controlling them by associating a suitably normalized submartingale sequence.
\end{itemize}
\begin{rmk}
To prove \eqref{Adamresult}, one moves from the full partial sums of $f$ to sums over integers with a large prime factor. Unfortunately, it is not possible to go the other way round. In particular, Theorem \ref{newthm} does not lead to an almost sure upper bound for the full partial sums of $f$. Indeed, one should estimate also the complementary portion over $\sqrt{x}$-smooth numbers (i.e., numbers $n$ with $P(n)\leq \sqrt{x}$), which requires exploiting more the intricated dependence between the values of $f(n)$, that cannot now be straightforwardly removed with a conditioning argument. 
\end{rmk}
\section{Preliminary results}
\subsection{Probabilistic number theoretic results}
As crucial in Lau--Tenenbaum--Wu's paper and previous works, we will need a control on the $2m$-th moment of weighted sums of random multiplicative functions. The following lemma allows us to do so by shifting the problem to computing the $L_2$-norm of a sum of such weights. Because $m$ can be arbitrary, this explains the name of such a result (see e.g. Harper's paper \cite[Proof of Probability Result 1]{H4}).
\begin{lem}(Hypercontractive inequality).
\label{lemhyperineq}
Let $f$ be a Rademacher or Steinhaus random multiplicative function. For any sequence $(a_n)_{n=1}^{+\infty}$ of complex numbers and any positive integer $m\geq 1$, we have
\begin{align*}
\mathbb{E}\bigg[\bigg| \sum_{n\geq 1}a_n f(n)\bigg|^{2m}\bigg]\leq \bigg(\sum_{n\geq 1} |a_n|^2 d_{2m-1}(n)\bigg)^{m},
\end{align*}
where for any $m\geq 1$, $d_m(n)$ is the $m$-fold divisor function.
\end{lem}
To compute the resulting sums in Lemma \ref{lemhyperineq}, we will make use of the following standard bound on the partial sums of divisor functions.
\begin{lem}
\label{lemmapartialsumdiv}
Let $M>0$. Then, uniformly for positive integers $m\leq M$ and $x\geq 2$, one has
\begin{align*}
\sum_{n\leq x}d_m(n)\ll x(\log x)^{m-1}.
\end{align*}
\end{lem}
\begin{proof}
This easily follows from \cite[Ch. III, Corollary 3.6]{T}.
\end{proof}
By proceeding similarly as in Harper \cite[Sect.
2.5]{H3} and previously as in Harper, Nikeghbali and Radziwi\l\l \cite[Sect. 2.2]{HNR}, we will smoothen certain partial sums of random multiplicative functions to replace them with an integral of a Dirichlet series corresponding to such functions. To this aim, we will need the following version of Parseval's identity for Dirichlet series.
\begin{lem}[Parseval's identity] 
\label{harmonicanalysisres}
Let $(a_n)_{n=1}^{+\infty}$ be any sequence of complex numbers and let $A(s):= \sum_{n=1}^{+\infty}\frac{a_n}{n^s}$ denote the
corresponding Dirichlet series and let also $\sigma_c$ denote its abscissa of convergence. Then for any $\sigma>\max\{0, \sigma_c\}$, we have
\begin{align*}
\int_{0}^{+\infty}\frac{|\sum_{n\leq x} a_n|^2}{x^{1+2\sigma}}dx=\frac{1}{2\pi }\int_{-\infty}^{+\infty}\bigg| \frac{A(\sigma+it)}{\sigma+it}\bigg|^2 dt.
\end{align*}
\end{lem}
\begin{proof}
This is \cite[Eq. $(5.26)$]{MV}.
\end{proof}
We will apply Lemma \ref{harmonicanalysisres} to the sequence given by a random multiplicative function. By multiplicativity,  its Dirichlet series can be recast in terms of an Euler product, for which we then need an $L^2$-estimate.
\begin{lem}[Euler product result]
\label{Eulerprodlem}
If $f$ is a Rademacher random multiplicative function, then for any real numbers $t$ and $2\leq  x \leq  y$, we have
\begin{align*}
\mathbb{E}\bigg[ \prod_{x<p\leq y}\bigg| 1+\frac{f(p)}{p^{1/2+it}}\bigg|^2\bigg]=\prod_{x<p\leq y}\bigg( 1+\frac{1}{p}\bigg).
\end{align*}
When $f$ is a Steinhaus random multiplicative function, we instead have
\begin{align*}
\mathbb{E}\bigg[ \prod_{x<p\leq y}\bigg| 1-\frac{f(p)}{p^{1/2+it}}\bigg|^{-2}\bigg]=\prod_{x<p\leq y}\bigg( 1-\frac{1}{p}\bigg)^{-1}.
\end{align*}
\end{lem}
\begin{proof}
Let $f$ be a Rademacher random multiplicative function. By the independence of the $f(p)$'s, for different prime numbers $p$, we get
\begin{align*}
\mathbb{E}\bigg[ \prod_{x<p\leq y}\bigg| 1+\frac{f(p)}{p^{1/2+it}}\bigg|^2\bigg]=\prod_{x<p\leq y}\mathbb{E}\bigg[\bigg| 1+\frac{f(p)}{p^{1/2+it}}\bigg|^2 \bigg].
\end{align*}
Expanding the square gives
\begin{align*}
\bigg| 1+\frac{f(p)}{p^{1/2+it}}\bigg|^2=1+\frac{f(p)}{p^{1/2+it}}+\frac{f(p)}{p^{1/2-it}}+\frac{1}{p}.
\end{align*}
Since $\mathbb{E}[f(p)]=0$, for any prime $p$, it is immediate to get the thesis.

Let now $f$ be a Steinhaus random multiplicative function and, for any prime number $p$, write
\begin{align*}
\bigg(1-\frac{f(p)}{p^{1/2+it}}\bigg)^{-1}=\sum_{\substack{k\geq 0}} \frac{f(p^k)}{p^{k(1/2+it)}}.
\end{align*}
Then, we clearly have
\begin{align*}
\mathbb{E}\bigg[\bigg|1-\frac{f(p)}{p^{1/2+it}}\bigg|^{-2} \bigg]=\mathbb{E}\bigg[\sum_{\substack{k\geq 0}} \frac{f(p^k)}{p^{k(1/2+it)}}\sum_{\substack{j\geq 0}} \frac{\overline{f(p^j)}}{p^{j(1/2-it)}} \bigg]=\sum_{\substack{k\geq 0}} \frac{1}{p^k},
\end{align*}
where we can exchange summations and expectation thanks to Tonelli--Fubini's theorem. By the independence of the $f(p)$'s, for different prime numbers $p$, we deduce the thesis also in this case.
\end{proof}
\subsection{Pure probabilistic results}
Common tools to tackle Law of the Iterated Logarithm type results, both classically and related to random multiplicative functions, are the well-known Borel--Cantelli's lemmas. Since the values of a random multiplicative function are not all independent one another, so do many of the events we will have to deal with; hence, we are interested only in applications of the first Borel--Cantelli's lemma (see e.g. \cite[Theorem 18.1]{G}).
\begin{lem}(The first Borel--Cantelli's lemma)
\label{borelcantelli}
Let $\{A_n\}_{n\geq 1}$ be any sequence of events. Then
\begin{align*}
\sum_{n=1}^{+\infty}\mathbb{P}(A_n)<+\infty \Rightarrow \mathbb{P}(\limsup_{n\rightarrow +\infty} A_n)=0,
\end{align*}
where
\begin{align*}
\limsup_{n\rightarrow +\infty} A_n:=\bigcap_{n=1}^{+\infty}\bigcup_{m=n}^{+\infty} A_m.
\end{align*}
\end{lem}
The next result is the celebrated Hoeffding's inequality, which gives Gaussian-type tails for the probability that a sum of many bounded \emph{independent} random variables deviates from its mean value by more than a certain amount (see e.g. Hoeffding \cite[Theorem 2]{HO}).
\begin{lem}[Hoeffding's inequality]
\label{Hoeffdinglem}
Let $X_1,\dots, X_n$ be independent random variables bounded by the intervals $[a_i,b_i]$. Let $S_n=X_1+\cdots+X_n$. Then we have
\begin{align*}
\mathbb{P}(|S_n-\mathbb{E}[S_n]|\geq t)\leq 2\exp\bigg(-\frac{2t^2}{\sum_{i=1}^{n}(b_i-a_i)^2}\bigg).
\end{align*}
\end{lem}
The next lemma gives a strong uniform control on the supremum of a finite number of terms in a submartingale sequence (see e.g. \cite[Theorem 9.1]{G}).
\begin{lem}[Doob's maximal inequality]
\label{Doobmaxineq}
Let $\lambda>0$. Suppose that the sequence of random variables and of $\sigma$-algebras $\{(X_n,\mathcal{F}_n)\}_{n\geq 0}$ is a nonnegative submartingale. Then
\begin{align*}
\lambda\mathbb{P}(\max_{0\leq k\leq n}X_k>\lambda)\leq \mathbb{E}[X_n].
\end{align*}
\end{lem}
On the other hand, regarding the moments of such supremum, we have the following result (see e.g. \cite[Theorem 9.4]{G}).
\begin{lem}[Doob's $L^p$-inequality]
\label{DoobLpineq}
Let $p>1$. Suppose that the sequence of random variables and of $\sigma$-algebras $\{(X_n,\mathcal{F}_n)\}_{n\geq 0}$ is a nonnegative submartingale bounded in $L^p$. Then
\begin{align*}
\mathbb{E}[(\max_{0\leq k\leq n}X_k)^p]\leq \bigg(\frac{p}{p-1}\bigg)^p \max_{0\leq k\leq n}\mathbb{E}[X_k^p].
\end{align*}
\end{lem}
\section{Proof of Theorem \ref{newthm}: setting up the argument}
Let $\eps>0$ and define
$$M_f(x):=\sum_{\substack{n\leq x\\ P(n)>\sqrt{x}}}f(n).$$
We would like to show that the event
$$\mathcal{A}:=\bigg\{|M_f(x)|> 6\sqrt{x}(\log\log x)^{1/4+\eps},\ \text{for infinitely many $x$}\bigg\},$$
holds with null probability.

As in Basquin \cite{BA} and in Lau--Tenenbaum--Wu \cite{LTW}, we are going to check the condition of the event $\mathcal{A}$ on a suitable sequence of test points $x_i$, not too much sparse to yet guarantee enough control on the size of $M_f(x)$ between two consecutive such points. As in the aforementioned works, we take $x_i:=\lfloor e^{i^{\eps}}\rfloor $. Moreover, again following previous arguments, we are going to focus our analysis on the test points contained in very wide intervals $[X_{\ell-1},X_\ell]$ so that $\mathcal{A}\subset\cup_{\ell\geq 1} \mathcal{A}_\ell$, where
\begin{align*}
\mathcal{A}_\ell:=\left\{ \sup_{X_{\ell-1}<x_{i-1}\leq X_\ell}\sup_{x_{i-1}<x\leq x_i} \frac{|M_f(x)|}{\sqrt{x}R(x)}>6\right\}
\end{align*}
and where, for the sake of readability, we let $R(x):=(\log\log x)^{1/4+\eps}$.
We here choose $X_\ell:=e^{2^{\ell^K}}$, with $K:=1/(4\eps)$. Unlike in \cite{BA} and \cite{LTW}, where $K$ was equal to $1$, we will work with an extremely sparser sequence $X_\ell$. We then note that
\begin{align*}
2^{(\ell-1)^K}<\log x_{i-1}\leq 2^{\ell^K}&\Rightarrow (\ell-1)^K \log 2<\log\log x_{i-1}\leq \ell^K \log 2\\
&\Rightarrow \log\log x_i \sim \ell^K \log 2,\ \text{as $\ell\rightarrow +\infty$},
\end{align*}
for any $x_{i-1}\in [X_{\ell-1}, X_\ell]$.

For any $x\in [x_{i-1}, x_i]$, we may write 
$$M_f(x)=M_f(x_{i-1})+(M_f(x)-M_f(x_{i-1})).$$
Hence, we get
\begin{align*}
|M_f(x)|\leq |M_f(x_{i-1})|+\bigg| \sum_{\substack{n\leq x_{i-1}\\ \sqrt{x_{i-1}}<P(n)\leq \sqrt{x}}}f(n)\bigg|+\bigg| \sum_{\substack{x_{i-1}<n\leq x\\ P(n)>\sqrt{x}}}f(n)\bigg|.
\end{align*}
Since the function $\sqrt{x}R(x)$ is an increasing function of $x$, we see that $\mathcal{A}_\ell\subset\mathcal{B}_\ell\cup \mathcal{C}_\ell\cup \mathcal{D}_\ell$, where
\begin{align*}
\mathcal{B}_\ell:&=\left\{ \sup_{X_{\ell-1}<x_i\leq X_\ell} \frac{|M_f(x_i)|}{\sqrt{x_i}R(x_i)}>2\right\}\\
\mathcal{C}_\ell:&=\left\{ \sup_{X_{\ell-1}<x_{i-1}\leq X_\ell}\frac{1}{\sqrt{x_{i-1}}R(x_{i-1})}\sup_{x_{i-1}<x\leq x_i}\bigg| \sum_{\substack{n\leq x_{i-1}\\ \sqrt{x_{i-1}}<P(n)\leq \sqrt{x}}}f(n)\bigg|>2\right\}\\
\mathcal{D}_\ell:&=\left\{ \sup_{X_{\ell-1}<x_{i-1}\leq X_\ell}\frac{1}{\sqrt{x_{i-1}}R(x_{i-1})}\sup_{x_{i-1}<x\leq x_i}\bigg| \sum_{\substack{x_{i-1}<n\leq x\\ P(n)>\sqrt{x}}}f(n)\bigg|>2\right\}.
\end{align*}
The event $\mathcal{B}_\ell$ encodes information about the size of $M_f(x)$ on test points. On the other hand, the events $\mathcal{C}_\ell$ and $\mathcal{D}_\ell$ together control the size of the increments of the partial sums of $f$ between two consecutive test points $x_{i-1},x_i$.

Now, suppose that we are given upper bounds $B_\ell, C_\ell$ and $D_\ell$ on the probabilities of $\mathcal{B}_\ell, \mathcal{C}_\ell$ and $\mathcal{D}_\ell$. If $\sum_{\ell\geq 1}(B_\ell+C_\ell+D_\ell)<+\infty$, so is $\sum_{\ell\geq 1}\mathbb{P}(\mathcal{A}_\ell)$. By the first Borel--Cantelli's lemma, Lemma \ref{borelcantelli}, we would then deduce that
$
\mathbb{P}(\limsup_{\ell\geq 1} \mathcal{A}_\ell )=0,
$
which, in turn, implies that, for any sufficiently large $x$, we would almost surely have $|M_f(x)|\leq 6\sqrt{x}(\log\log x)^{1/4+\eps}$, which is the content of Theorem \ref{newthm}.
\section{The sum between test points}
The aim in this section is to find a bound summable on $\ell$ for the probability of the events $\mathcal{C}_\ell$ and $\mathcal{D}_\ell$. As explained in \cite{LTW}, this would show that, almost surely, the partial sum $M_f(x)$ fluctuates moderately in appropriate short intervals and that the problem of bounding $M_f(x)$ everywhere may be reduced to doing so at the suitable test points $x_i$.
\subsection{The probability of $\mathcal{C}_\ell$}
By the union bound, the probability of $\mathcal{C}_\ell$ is
\begin{align*}
&\leq \sum_{X_{\ell-1}<x_{i-1}\leq X_\ell}\mathbb{P}\bigg(\sup_{x_{i-1}<x\leq x_i}\bigg| \sum_{\substack{n\leq x_{i-1}\\ \sqrt{x_{i-1}}<P(n)\leq \sqrt{x}}}f(n)\bigg|>2\sqrt{x_{i-1}}R(x_{i-1})\bigg)\\
&=\sum_{X_{\ell-1}<x_{i-1}\leq X_\ell}\mathbb{P}\bigg(\sup_{x_{i-1}<x\leq x_i}\bigg| \sum_{\substack{n\leq x_{i-1}\\ \sqrt{x_{i-1}}<P(n)\leq \sqrt{x}}}f(n)\bigg|^2>4x_{i-1}R(x_{i-1})^2\bigg).
\end{align*}
By Markov's inequality for the power $2$, the above is
\begin{align*}
\ll \sum_{X_{\ell-1}<x_{i-1}\leq X_\ell}\frac{1}{x_{i-1}^2R(x_{i-1})^4}\mathbb{E}\bigg[ \bigg(\sup_{x_{i-1}<x\leq x_i}\bigg|\sum_{\substack{n\leq x_{i-1}\\ \sqrt{x_{i-1}}<P(n)\leq \sqrt{x}}}f(n)\bigg|^2\bigg)^2 \bigg].
\end{align*}
Now, consider the sequence of random variables $(Z_k)_{k\geq 1}$ given by
\begin{align*}
Z_{k}:=\bigg| \sum_{\substack{n\leq x_{i-1}\\ \sqrt{x_{i-1}}<P(n)\leq \lfloor \sqrt{k}\rfloor}}f(n)\bigg|^2.
\end{align*}
To move from one element of the sequence $Z_k$ to the next, we reveal at most one new prime at a time. This usually corresponds to having a submartingale structure. In fact, $(Z_k)_{k\geq 1}$ does form a nonnegative submartingale with respect to the filtration $\mathcal{F}_k:=\sigma(\{f(p): p\leq \lfloor \sqrt{k}\rfloor\})$. Indeed, $Z_k$ is clearly $\mathcal{F}_k$-adapted and $L^1$-bounded and  furthermore
\begin{align*}
\mathbb{E}[Z_{k+1}|\mathcal{F}_k]&=Z_k+\mathbb{E}\bigg[\bigg|\sum_{\substack{n\leq x_{i-1}\\ \lfloor \sqrt{k}\rfloor <P(n)\leq \lfloor \sqrt{k+1}\rfloor}}f(n)\bigg|^2\bigg|\mathcal{F}_k\bigg]\\
&+2\Re\bigg(\sum_{\substack{n\leq x_{i-1}\\ \sqrt{x_{i-1}}<P(n)\leq \lfloor \sqrt{k}\rfloor}}\overline{f(n)}\mathbb{E}\bigg[\sum_{\substack{n\leq x_{i-1}\\ \lfloor \sqrt{k}\rfloor <P(n)\leq \lfloor \sqrt{k+1}\rfloor}}f(n)\bigg|\mathcal{F}_k\bigg]\bigg)\\
&\geq Z_k,
\end{align*}
because for any $n$ in the innermost sum on the second line above we have $f(n)=f(p)f(m)$, with $\lfloor \sqrt{k}\rfloor <p\leq \lfloor \sqrt{k+1}\rfloor $ and $m$ divided only by primes smaller than $\lfloor \sqrt{k}\rfloor$, so that $\mathbb{E}[f(n)|\mathcal{F}_k]=f(m)\mathbb{E}[f(p)]=0.$  

Whence, an application of Doob's $L^2$-inequality, Lemma \ref{DoobLpineq}, leads to a bound for the probability of $\mathcal{C}_\ell$
\begin{align*}
\ll \sum_{X_{\ell-1}<x_{i-1}\leq X_\ell}\frac{1}{x_{i-1}^2R(x_{i-1})^4}\sup_{x_{i-1}<x\leq x_i}\mathbb{E}\bigg[\bigg|\sum_{\substack{n\leq x_{i-1}\\ \sqrt{x_{i-1}}<P(n)\leq \sqrt{x}}}f(n)\bigg|^4 \bigg].
\end{align*}
To compute the fourth moment we appeal to Lemma \ref{lemhyperineq}, which gives a bound
\begin{align*}
&\leq \bigg(\sum_{\substack{n\leq x_{i-1}\\ \sqrt{x_{i-1}}<P(n)\leq \sqrt{x_i}}}d_3(n)\bigg)^2.
\end{align*}
Finally, we write
\begin{align*}
\sum_{\substack{n\leq x_{i-1}\\ \sqrt{x_{i-1}}<P(n)\leq \sqrt{x_i}}}d_3(n)=3\sum_{\sqrt{x_{i-1}}<p\leq \sqrt{x_i}}\sum_{k\leq x_{i-1}/p}d_3(k)
\end{align*}
and estimate the divisor sum on the right-hand side of the previous displayed equation by using Lemma \ref{lemmapartialsumdiv}. In this way, we get an overall bound for $\mathbb{P}(\mathcal{C}_\ell)$
\begin{align*}
&\ll \sum_{X_{\ell-1}<x_{i-1}\leq X_\ell}\frac{1}{x_{i-1}^2R(x_{i-1})^4}x_{i-1}^2(\log x_{i-1})^4\bigg(\sum_{\sqrt{x_{i-1}}<p\leq \sqrt{x_i}}\frac{1}{p}\bigg)^2\\
&\ll \sum_{X_{\ell-1}<x_{i-1}\leq X_\ell}\frac{(\log x_{i-1})^4}{i^2R(x_{i-1})^4}\ll \sum_{i\geq 2^{\frac{(\ell-1)^K}{\eps}}}\frac1{i^{2-4\eps}}\ll 2^{-(\ell-1)^K (1-4\eps)/\eps},
\end{align*}
by a strong form of Mertens' theorem (with error term given by the Prime Number Theorem), if $\ell$ is sufficiently large with respect to $\eps$. This is certainly summable on $\ell$, if $\eps<1/4.$
\subsection{The probability of $\mathcal{D}_\ell$} By the union bound 
\begin{align*}
\mathbb{P}(\mathcal{D}_l)\leq \sum_{X_{\ell-1}<x_{i-1}\leq X_\ell}\mathbb{P}\bigg(\sup_{x_{i-1}<x\leq x_i}\bigg| \sum_{\substack{x_{i-1}<n\leq x\\ P(n)>\sqrt{x}}}f(n)\bigg|>2\sqrt{x_{i-1}}R(x_{i-1})\bigg).
\end{align*}
The probability of the above event where instead the partial sum of $f(n)$ runs over the full short interval $[x_{i-1},x]$ has already been studied by Basquin \cite{BA} and Lau--Tenenbaum--Wu \cite{LTW}. Here, we have to deal with the extra condition on the largest prime factor, which can still be handled by adapting the proof in the aforementioned papers.

We split $[x_{i-1}, x_i]$ into a disjoint union
of at most $2\log x_i$ subintervals with limit points
\begin{align*}
u_k:=x_{i-1}+\sum_{1\leq j\leq k}2^{\nu_j}\ \ \ (0\leq k\leq h),
\end{align*}
with $\nu_1>\nu_2>\cdots>\nu_h$ positive integers.

Then, by seeing 
$$S_f(x):= \sum_{\substack{x_{i-1}<n\leq x\\ P(n)>\sqrt{x}}}f(n)=S_f(x)-S_f(x_{i-1}),$$
we can bound the above probability with
\begin{align*}
&\leq \mathbb{P}\bigg(\sum_{u_k}|S_f(u_{k+1})-S_f(u_k)|>2\sqrt{x_{i-1}}R(x_{i-1})\bigg)\\
&\leq \mathbb{P}\bigg(\bigcup_{u_k}\bigg\{|S_f(u_{k+1})-S_f(u_k)|>\frac{\sqrt{x_{i-1}}R(x_{i-1})}{\log x_i}\bigg\}\bigg)\\
&=\mathbb{P}\bigg(\sup_{u_k} |S_f(u_{k+1})-S_f(u_k)|>\frac{\sqrt{x_{i-1}}R(x_{i-1})}{\log x_i}\bigg).
\end{align*} 
Moreover, note that for any $x_{i-1}\leq u\leq v\leq x_i$, we have
\begin{align*}
S_f(v)-S_f(u)=-\sum_{\substack{x_{i-1}<n\leq u\\ \sqrt{u}<P(n)\leq\sqrt{v}}}f(n)+\sum_{\substack{u<n\leq v\\ P(n)>\sqrt{v}}}f(n).
\end{align*}
Now, write $u=x_{i-1}+(l-1)2^m$ and $v=x_{i-1}+l2^m$, where $l:=\sum_{1\leq j\leq k}2^{\nu_j-\nu_k}\geq 1$ and $m:=\nu_k\geq 0$ are such that $l2^m\leq x_i-x_{i-1}$. 

By the union bound, Markov's inequality for the fourth moment and the hypercontractive inequality as stated in Lemma \ref{lemhyperineq}, we have a bound for the probability of $\mathcal{D}_\ell$
\begin{align}
\label{boundDl1}
&\ll \sum_{X_{\ell-1}<x_{i-1}\leq X_\ell}\frac{(\log x_i)^4}{x_{i-1}^2}\sum_{\substack{l\geq 1, m\geq 0\\ l2^m\leq x_i-x_{i-1}}}\bigg(\sum_{u<n\leq v} d_3(n)\bigg)^2\\
\label{boundDl2}
&+\sum_{X_{\ell-1}<x_{i-1}\leq X_\ell}\frac{(\log x_i)^4}{x_{i-1}^2}\sum_{\substack{l\geq 1, m\geq 0\\ l2^m\leq x_i-x_{i-1}}}\bigg(\sum_{\substack{x_{i-1}<n\leq u\\ \sqrt{u}<P(n)\leq\sqrt{v}}} d_3(n)\bigg)^2.
\end{align}
By H\"older's inequality, we get
\begin{align*}
\bigg(\sum_{u<n\leq v} d_3(n)\bigg)^2&\leq \bigg(\sum_{n\leq v} d_3(n)^3\bigg)^{2/3}\bigg(\sum_{u<n\leq v}1\bigg)^{4/3}\\
&\ll x_i^{2/3}(\log x_i)^{52/3}(v-u)^{4/3},
\end{align*}
where the estimate for the partial sum of $d_3(n)^3$ easily follows from \cite[Ch. III, Corollary 3.6]{T}.  Summing this up over all possible realizations of $u$ and $v$, we get an overall bound for \eqref{boundDl1} of
\begin{align}
\label{boundDl11}
&\ll \sum_{X_{\ell-1}<x_{i-1}\leq X_\ell}(\log x_i)^{64/3} \bigg(\frac{x_i-x_{i-1}}{x_i}\bigg)^{4/3}\\
&\ll \sum_{X_{\ell-1}<x_{i-1}\leq X_\ell}\frac{1}{i^{4/3-68\eps/3}}\nonumber\\
&\ll 2^{-(\ell-1)^K(1/3-68\eps/3)/\eps}.\nonumber
\end{align}
Regarding \eqref{boundDl2}, we notice that 
\begin{align}
\label{doublesum}
\sum_{\substack{x_{i-1}<n\leq u\\ \sqrt{u}<P(n)\leq\sqrt{v}}} d_3(n)=3 \sum_{\sqrt{u}<p\leq \sqrt{v}}\sum_{\frac{x_{i-1}}{p}<k\leq \frac{u}{p}}d_3(k).
\end{align}
If $\sqrt{v}-\sqrt{u}\geq 1$, we simply upper bound the innermost sum on the right-hand side above with $\ll u(\log u)^2/p$, by Lemma \ref{lemmapartialsumdiv}, and get a bound for \eqref{doublesum} of
\begin{align*}
\ll u(\log u)^2\sum_{\sqrt{u}<p\leq \sqrt{v}}\frac{1}{p}&\ll \sqrt{u}(\log u)^2(\sqrt{v}-\sqrt{u})\\
&\leq (v-u)(\log x_i)^2.
\end{align*}
This contibutes to \eqref{boundDl2} an amount of
\begin{align}
\label{estimate1}
&\ll\sum_{X_{\ell-1}<x_{i-1}\leq X_\ell}\frac{(\log x_i)^8}{x_{i-1}^2}\sum_{\substack{l\geq 1, m\geq 0\\ l2^m\leq x_i-x_{i-1}}}(v-u)^2\\
&\ll \sum_{X_{\ell-1}<x_{i-1}\leq X_\ell}\frac{(\log x_i)^8(x_i-x_{i-1})^2}{x_{i-1}^2}\nonumber\\
&\ll \sum_{X_{\ell-1}<x_{i-1}\leq X_\ell}\frac{1}{i^{2-10\eps}}\nonumber\\
&\ll 2^{-(\ell-1)^K(1-10\eps)/\eps}.\nonumber
\end{align}
On the other hand, if $\sqrt{v}-\sqrt{u}<1$, we extend the innermost sum on the right-hand side of \eqref{doublesum} to all the integers in the interval $[x_{i-1}/p, x_i/p]$ to then, by Shiu's theorem \cite[Theorem 1]{SH}, upper bound it with 
$$\frac{(x_i-x_{i-1})(\log x_i)^2}{p}\leq \frac{(x_i-x_{i-1})(\log x_i)^2}{\sqrt{x_{i-1}}}.$$
The application of Shiu's theorem is justified by the fact that
\begin{align*}
\frac{x_i}{p}-\frac{x_{i-1}}{p}>\sqrt[3]{\frac{x_{i}}{p}},
\end{align*}
if $x_i$ is sufficiently large with respect to $\eps$, as it can be easily verified.
This bound contributes to \eqref{boundDl2} an amount of
\begin{align}
\label{estimate2}
&\ll\sum_{X_{\ell-1}<x_{i-1}\leq X_\ell} \frac{(x_i-x_{i-1})^3(\log x_i)^8}{x_{i-1}^3}\\
&\ll\sum_{X_{\ell-1}<x_{i-1}\leq X_\ell} \frac{1}{i^{3-11\eps}}\nonumber\\
&\ll 2^{-(\ell-1)^K(2-11\eps)/\eps}.\nonumber
\end{align}
Together, the estimates \eqref{boundDl11}, \eqref{estimate1} and \eqref{estimate2} give a total bound for the probability of $\mathcal{D}_\ell$ that is summable on $\ell$, if $\eps$ is small enough.
\section{The sum on test points and conditional conclusion of the proof of Theorem \ref{newthm}}
Thanks to the work done in the previous sections, to prove Theorem \ref{newthm}, we are left with understanding the size of the partial sums of $f$ over the test points $x_i$. More specifically, we need to bound the probability of the following event
\begin{align*}
\mathcal{B}_\ell:&=\left\{ \sup_{X_{\ell-1}<x_i\leq X_\ell} \frac{|M_f(x_i)|}{\sqrt{x_i}R(x_i)}>2\right\}.
\end{align*}
Assume $f$ to be a Rademacher random multiplicative function. To this aim, we first notice that we may rewrite the partial sums of $f$ over integers with a large prime factor as a sum of many independent random variables, if we allow for conditioning on the smaller primes. In fact, 
$$M_f(x_i)=\sum_{\sqrt{x_i}<p\leq x_i} Y_p,$$
where, for any $p>\sqrt{x_i}$, we let
\begin{align*}
Y_p:=f(p)\sum_{m\leq x_i/p}f(m).
\end{align*}
The random variables $(Y_p)_{\sqrt{x_i}<p\leq x_i}$,
conditioned on 
$$\mathcal{F}(\sqrt{x_i}):=\sigma(\{f(p):\ p\leq \sqrt{x_i}\}),$$ are independent, with $\mathbb{E}[Y_p | \mathcal{F}(\sqrt{x_i})]=0.$

We are then in position to apply Hoeffding's inequality, Lemma \ref{Hoeffdinglem}, to get
\begin{align}
\label{Hoeffdingineq}
\mathbb{P}\bigg(|M_f(x_i)|\geq 2\sqrt{x_i}R(x_i)\ | \mathcal{F}(\sqrt{x_i})\bigg)&\ll\exp\bigg(-\frac{4 x_i R(x_i)^2}{V(x_i)}\bigg),
\end{align}
where
\begin{equation}
V(x_i):=\sum_{\sqrt{x_i}<p\leq x_i} \bigg| \sum_{m\leq x_i/p}f(m)\bigg|^2.
\end{equation}
We arrive to the same bound \eqref{Hoeffdingineq}, where the constants $2$ and $4$ are replaced by $1$, if, for a Steinhaus random multiplicative function, we replace $M_f(x)$ with $\Re(M_f(x))$ and $\Im(M_f(x))$. 

Clearly, the right-hand side of \eqref{Hoeffdingineq} is still a random variable. However, if we condition on the size of $V(x_i)$, it will lead to an estimate for the probability of $\mathcal{B}_\ell$. To this regard, we will show that with high probability (depending on $\ell$)
\begin{align}
\label{heuristicvariance}
V(x_i)\ll \frac{x_i}{\sqrt{\log\log x_i}},
\end{align}
uniformly on $x_i\in [X_{\ell-1}, X_\ell]$. The scaling factor $\sqrt{\log\log x_i}$, compared to $\mathbb{E}[V(x_i)]\asymp x_i$, is characteristic of the low moments of partial sums of Rademacher and Steinhaus random multiplicative functions (see the introduction to \cite{H3}) and we can already pointwise derive \eqref{heuristicvariance} using Harper's low moments results \cite{H3}. The uniformity in \eqref{heuristicvariance} will come from rewriting $V(x_i)$ in terms of a submartingale sequence and managing its size via Doob's inequality. These features, together with the Gaussian-type control \eqref{Hoeffdingineq} on the tail distribution of $M_f(x_i)$, are what determines the exponent $1/4$ in Theorem \ref{newthm}.

Following previous considerations, we define
$$\mathcal{E}_\ell:=\left\{\sup_{X_{\ell-1}<x_i\leq X_\ell}  \frac{V(x_i)\sqrt{\log\log x_i}}{x_i}\leq T\right\},$$
with $T\geq 1$ a parameter that will be chosen later and that measures how much we can lose, compared to \eqref{heuristicvariance}, to still be able to successfully estimate the probability of $\mathcal{B}_\ell$.
We now show how to deduce Theorem \ref{newthm} from the next lemma. 
\begin{lem}
\label{lemS1}
Let $\eps>0$ and $K=1/(4\eps)$. Then, for any $T\geq 1$, we have
$$\mathbb{P}(\bar{\mathcal{E}_\ell}) \ll_\eps \bigg(\frac{\sqrt{\ell^K}}{T}\bigg)^{\frac{4}{\eps}}2^{-\frac{\ell^K}{\eps}}+\frac{1}{T^{1/4}}.$$
\end{lem}
\begin{proof}[Conclusion of the proof of Theorem \ref{newthm}, given Lemma \ref{lemS1}]
Let $f$ be a Rademacher random multiplicative function. Plainly, 
\begin{align*}
\mathbb{P}(\mathcal{B}_\ell)&=\mathbb{P}\bigg(\bigcup_{X_{\ell-1}<x_i\leq X_\ell}\bigg\{\frac{|M_f(x_i)|}{\sqrt{x_i}R(x_i)}>2\bigg\}\bigg)\\
&\leq\mathbb{P}\bigg(\bigcup_{X_{\ell-1}<x_i\leq X_\ell}\bigg\{\frac{|M_f(x_i)|}{\sqrt{x_i}R(x_i)}>2\bigg\}\cap\bigg\{\frac{V(x_i)\sqrt{\log\log x_i}}{x_i}\leq T\bigg\} \bigg)\\
&+\mathbb{P}\bigg(\bigcup_{X_{\ell-1}<x_i\leq X_\ell}\bigg\{\frac{|M_f(x_i)|}{\sqrt{x_i}R(x_i)}>2\bigg\}\cap\bigg\{\frac{V(x_i)\sqrt{\log\log x_i}}{x_i}> T\bigg\} \bigg).
\end{align*}
By a repeated application of \eqref{Hoeffdingineq}, where we condition on the event $\{V(x_i)\sqrt{\log\log x_i}\leq Tx_i\}$, and the union bound, we get the above is
\begin{align*}
\ll \sum_{X_{\ell-1}<x_i\leq X_\ell}\exp\bigg(-\frac{4R(x_i)^2\sqrt{\log\log x_i}}{T}\bigg)+\mathbb{P}(\bar{\mathcal{E}_\ell}). 
\end{align*}
We arrive to the same above bound, by replacing $M_f(x)$ with $\Re(M_f(x))$ and $\Im(M_f(x))$ and the constants $2$ and $4$ with $1$, in the Steinhaus case.
 
Remind that $R(x_i)=(\log \log x_i)^{1/4+\eps}\sim \ell^{K/4+K\eps}(\log 2)^{1/4+\eps}$, with $K\eps=4$. Whence, by letting $T=T(\ell):=\eps^2 \ell^8$, the above becomes
\begin{align*}
\ll \sum_{X_{\ell-1}<x_i\leq X_\ell}\exp\bigg(-\frac{c\ell^K}{\eps^2}\bigg)+\mathbb{P}(\bar{\mathcal{E}_\ell}),
\end{align*} 
where $c>0$ is a constant. By Lemma \ref{lemS1}, we overall deduce that
\begin{align*}
\mathbb{P}(\mathcal{B}_\ell)\ll_\eps 2^{\frac{\ell^K}{\eps}}e^{-\frac{c\ell^K}{\eps^2}}+2^{-\frac{\ell^K}{\eps}}\ell^{8/\eps^2-32/\eps}+\ell^{-2},
\end{align*}
which is evidently a bound summable on $\ell$, if $\eps$ is taken small enough. Since the same holds for the probabilities of $\mathcal{C}_\ell$ and $\mathcal{D}_\ell$, as proved in subsect. $4.1$ and $4.2$, we overall get a summable bound for the probability of $\mathcal{A}_\ell$, thus concluding the proof of Theorem \ref{newthm}, in the way it was described at the end of sect. $3$.
\end{proof}
\begin{rmk}
We would like to stress how important is the introduction of the exponent $K$ in the definition of $X_\ell$. Even though it makes our task harder, by drastically increasing the number of test points $x_i$ contained between two consecutive elements of the sequence $\{X_\ell\}_{\ell \geq 1}$, it allows us to choose $T$ much \emph{smaller} compared to what is possible to do in Basquin \cite{BA} and Lau--Tenenbaum--Wu \cite{LTW}, leading to a superior bound in Theorem \ref{newthm}. Another source of saving in the choice of $T$ comes from inputting the new information about the low moments of the partial sums of $f$. 
\end{rmk}
\section{A smoothing argument}
In this section we start the proof of Lemma \ref{lemS1}. From now on, $f$ will indicate both a Rademacher and a Steinhaus random multiplicative function. To begin with, we make $V(x_i)$ more amenable to perform the computation of the probability of $\bar{\mathcal{E}_\ell}$. 

By taking inspiration from Harper's work \cite{H4} and previously from Harper, Nikeghbali and Radziwi\l\l s' work \cite{HNR}, we insert a logarithmic weight into the summation defining $V(x_i)$ to then smoothen the summation further by inserting an integral average. More specifically, we have
\begin{align*}
V(x_i)&\leq \frac{2}{\log x_i}\sum_{\sqrt{x_i}<p\leq x_i}\log p\bigg|\sum_{m\leq x_i/p}f(m)\bigg|^2\\
&=\frac{2X}{\log x_i}\sum_{\sqrt{x_i}<p\leq x_i}\frac{\log p}{p}\int_{p}^{p(1+1/X)}\bigg|\sum_{m\leq x_i/p}f(m)\bigg|^2dt\\
&\ll \frac{X}{\log x_i}\sum_{\sqrt{x_i}<p\leq x_i}\frac{\log p}{p}\int_{p}^{p(1+1/X)}\bigg|\sum_{\substack{m\leq x_i/t}}f(m)\bigg|^2dt\\
&+\frac{X}{\log x_i}\sum_{\sqrt{x_i}<p\leq x_i}\frac{\log p}{p}\int_{p}^{p(1+1/X)}\bigg|\sum_{x_i/t<m\leq x_i/p}f(m)\bigg|^2dt,
\end{align*}
by using $|a+b|^2\ll |a|^2+|b|^2,$ for any complex numbers $a$ and $b$, and where $X\geq 1$ will be chosen later.
Hence, the probability of $\bar{\mathcal{E}}_\ell$ may be bounded from above by $\mathbb{P}_1+\mathbb{P}_2$, where
\begin{align}
\label{smoothsum1}
\mathbb{P}_1:&=\mathbb{P}\bigg(\bigg\{\sup_{X_{\ell-1}<x_i\leq X_\ell}  \frac{X\sqrt{\log\log x_i}}{x_i\log x_i}\\
&\times \sum_{\sqrt{x_i}<p\leq x_i}\frac{\log p}{p}\int_{p}^{p(1+1/X)}\bigg|\sum_{\substack{m\leq x_i/t}}f(m)\bigg|^2dt> \frac{T}{2}\bigg\}\bigg)\nonumber
\end{align}
\begin{align}
\label{smoothsum2}
\mathbb{P}_2:&=\mathbb{P}\bigg(\bigg\{\sup_{X_{\ell-1}<x_i\leq X_\ell}  \frac{X\sqrt{\log\log x_i}}{x_i\log x_i}\\
&\times\sum_{\sqrt{x_i}<p\leq x_i}\frac{\log p}{p}\int_{p}^{p(1+1/X)}\bigg|\sum_{x_i/t<m\leq x_i/p}f(m)\bigg|^2dt> \frac{T}{2}\bigg\}\bigg).\nonumber
\end{align}
By the union bound and Markov's inequality for the power $q>1$, we have that
\begin{align}
\label{expectationsum1}
\mathbb{P}_2&\ll_q \frac{1}{T^q}\sum_{X_{\ell-1}<x_i\leq X_\ell} \bigg(\frac{\sqrt{\log\log x_i}}{x_i}\bigg)^q\\
&\times \mathbb{E}\bigg[ \bigg(\sum_{\sqrt{x_i}<p\leq x_i}\frac{X}{p}\int_{p}^{p(1+1/X)}\bigg|\sum_{x_i/t<m\leq x_i/p}f(m)\bigg|^2dt\bigg)^q\bigg].\nonumber
\end{align}
We fix $q:=4/\eps$, because we would like to roughly have $(\log x_i)^q$ of size comparable to the number of test points $x_i$ in $[X_{\ell-1},X_\ell]$.

The expectation above can be seen as the $q$th power of the $q$th norm of a sum of random variables. Then, it is natural to swap norm and summation, by appealing to Minkowski's inequality. We can thus bound such expectation with 
\begin{align*}
\leq \bigg(\sum_{\sqrt{x_i}<p\leq x_i}\bigg(\mathbb{E}\bigg[\bigg(\frac{X}{p}\int_{p}^{p(1+1/X)}\bigg|\sum_{x_i/t<m\leq x_i/p}f(m)\bigg|^2dt\bigg)^q \bigg] \bigg)^{\frac{1}{q}} \bigg)^{q}.
\end{align*}
The next step, arguing as in Harper \cite{H4}, is to switch the expectation with the integral. This is achieved by an application of H\"{o}lder's inequality to the normalised integral $\frac{X}{p}\int_p^{p(1+1/X)}dt$ with parameters $1/q$ and $(q-1)/q$. We then estimate the above with
\begin{align}
\label{firstboundexp}
\leq \bigg(\sum_{\sqrt{x_i}<p\leq x_i}\bigg(\frac{X}{p}\int_{p}^{p(1+1/X)}\mathbb{E}\bigg[\bigg|\sum_{x_i/t<m\leq x_i/p}f(m)\bigg|^{2q} \bigg]dt \bigg)^{\frac{1}{q}} \bigg)^{q}.
\end{align}
The problem is then reduced to bound the $2q$th moment of a partial sum of $f$ over short intervals. This is addressed by an application of the hypercontractive inequality. Indeed, arguing as in Harper \cite{H4}, we notice that if $x_i/(X+1)<p\leq x_i$, then 
$$\frac{X}{p}\int_{p}^{p(1+1/X)}\mathbb{E}\bigg[\bigg|\sum_{x_i/t<m\leq x_i/p}f(m)\bigg|^{2q} \bigg]dt \leq 1,$$
since the sum contains at most one element, having length $\frac{x_i(t-p)}{tp}<1$; otherwise, by again following Harper as in the proof of Proposition 2 in \cite{H4}, we apply the Cauchy--Schwarz's inequality to bound the expectation in \eqref{firstboundexp} with  
\begin{align*}
&\sqrt{\mathbb{E}\bigg[\bigg|\sum_{x_i/t<m\leq x_i/p}f(m)\bigg|^{2} \bigg]\mathbb{E}\bigg[\bigg|\sum_{x_i/t<m\leq x_i/p}f(m)\bigg|^{2(2q-1)}\bigg]}.
\end{align*}
Now, since $t<p(1+1/X)$, we clearly have
\begin{align*}
\mathbb{E}\bigg[\bigg|\sum_{x_i/t<m\leq x_i/p}f(m)\bigg|^{2} \bigg]\leq \sum_{x_i/(p(1+1/X))<m\leq x_i/p}1\ll \frac{x_i}{pX},
\end{align*}
where we used that $p\leq x_i/(X+1)$.
 
On the other hand, by Lemma \ref{lemhyperineq}, we find
\begin{align*}
\mathbb{E}\bigg[\bigg|\sum_{x_i/t<m\leq x_i/p}f(m)\bigg|^{2(2q-1)} \bigg]&\leq \bigg( \sum_{m\leq x_i/p}d_{4q-3}(m)\bigg)^{2q-1}\\
&\ll_q \bigg(\frac{x_i}{p}(\log x_i)^{4q-4}\bigg)^{2q-1},
\end{align*}
by Lemma \ref{lemmapartialsumdiv}.

Collecting the previous computations together, we have found 
\begin{align*}
\mathbb{E}\bigg[\bigg|\sum_{x_i/t<m\leq x_i/p}f(m)\bigg|^{2q}\bigg] \ll_q\bigg(\frac{x_i}{p}\bigg)^q\frac{(\log x_i)^{4q^2-6q+2}}{\sqrt{X}}.
\end{align*}
Hence, \eqref{firstboundexp} is
\begin{align*}
&\ll_q \bigg(\sum_{x_i/(X+1)<p\leq x_i}1+\sum_{\sqrt{x_i}<p\leq x_i/(X+1)}\frac{x_i}{pX^{1/2q}}(\log x_i)^{4q-6+2/q}\bigg)^q\\
&\ll_q \frac{x_i^q}{(\log x_i)^q},
\end{align*}
by choosing e.g. $X:=(\log x_i)^{8q^2-10q+4}$ and using estimates of Chebyshev and Mertens.
Inserting this back into \eqref{expectationsum1}, we deduce:
\begin{align*}
\mathbb{P}_2\ll_q \frac{1}{T^q}\sum_{X_{\ell-1}<x_i\leq X_\ell} \bigg(\frac{\sqrt{\log\log x_i}}{\log x_i}\bigg)^q&\ll_q \frac{2^{\frac{\ell^K}{\eps}}}{T^q} \bigg(\frac{\sqrt{\ell^K}}{2^{(\ell-1)^K}}\bigg)^q\\
&\leq \bigg(\frac{\sqrt{\ell^K}}{T}\bigg)^{\frac{4}{\eps}}2^{-\frac{\ell^K}{\eps}},
\end{align*}
reminding that $q=4/\eps$ and taking $\ell$ large enough with respect to $\eps$. This gives the first term in the upper bound of Lemma \ref{lemS1}.
\section{Inputting low moments estimates}
In this section we continue the proof of Lemma \ref{lemS1}, by now turning to the study of the probability in \eqref{smoothsum1}.
\subsection{Introducing a submartingale sequence}
Swapping integral and summation, we have
\begin{align*}
&X\sum_{\sqrt{x_i}<p\leq x_i}\frac{\log p}{p}\int_{p}^{p(1+1/X)}\bigg|\sum_{\substack{m\leq x_i/t}}f(m)\bigg|^2dt\\
&\leq X\int_{\sqrt{x_i}}^{x_i(1+1/X)}\sum_{t/(1+1/X)<p\leq t}\frac{\log p}{p}\bigg|\sum_{m\leq x_i/t}f(m)\bigg|^2dt.
\end{align*}
Since $\log t\asymp \log x_i$, and reminding that $X=(\log x_i)^{8q^2-10q+4}$, by a strong form of Mertens' theorem (with error term given by the Prime Number Theorem) we find 
$$\sum_{t/(1+1/X)<p\leq t}\frac{\log p}{p}\ll \log\bigg(1+\frac{1}{X}\bigg)\ll \frac{1}{X},$$
if $x_i$ is sufficiently large with respect to $\eps$.

Inserting the last estimate in the previous expression, and changing variables $x_i/t=:z$ inside the integral, we find it is
\begin{align}
\label{intermediateint}
\ll x_i\int_{0}^{\sqrt{x_i}}\bigg|\sum_{\substack{m\leq z}}f(m)\bigg|^2\frac{dz}{z^2}.
\end{align}
It will be soon clear that the above random variable generates a nonnegative submartingale sequence. This observation will help us out later to deal with a supremum of such sequence over the test points $x_i$, via the use of Doob's maximal inequality. However, an immediate application of such result would only lead to a too weak bound for $\mathbb{P}_1$. This is due to the fact that Doob's maximal inequality relates the probability of a supremum of a submartingale sequence only to the expectations of its members, not instead to their low moments (which we need here, because of the presence of the factors $\sqrt{\log\log x_i}$ in \eqref{smoothsum1}, which are related to the size of the low moments of the random variables in \eqref{intermediateint}). For similar reasons, even an application of Doob's $L^p$ inequality, Lemma \ref{DoobLpineq}, would be inefficient, considering that it only deals with high moments.
To overcome this, we will first condition on the event that the contribution from the values of $f$ on the small primes is dominated by the size of its low moments, and what follows goes in the direction of rewriting the integral in \eqref{intermediateint} in a way to make more accessible this kind of information.

By extending the integral in \eqref{intermediateint}, we find it is
\begin{align*}
=x_i\int_{0}^{\sqrt{x_i}}\bigg|\sum_{\substack{m\leq z\\ P(m)\leq x_i}}f(m)\bigg|^2\frac{dz}{z^2}\leq x_i \int_{0}^{+\infty}\bigg|\sum_{\substack{m\leq z\\ P(m)\leq x_i}}f(m)\bigg|^2\frac{dz}{z^{2}}.
\end{align*}
The idea of inserting the constraint on the largest prime factor is taken from the proof of \cite[Proposition 2]{H3}.
Continuing arguing as in there, by appealing to Parseval's identity, Lemma \ref{harmonicanalysisres}, we rewrite the above as
\begin{align*}
\frac{x_i}{2\pi}\int_{-\infty}^{+\infty} \bigg|\frac{\mathcal{S}_{x_i}(1/2+it)}{1/2+it}\bigg|^2dt,
\end{align*}
where
\begin{align*}
\mathcal{S}_{x_i}(1/2+it):=\prod_{p\leq x_i}\bigg(1+\frac{f(p)}{p^{1/2+it}}\bigg),
\end{align*}
in the Rademacher case, or
\begin{align*}
\mathcal{S}_{x_i}(1/2+it):=\prod_{p\leq x_i}\bigg(1-\frac{f(p)}{p^{1/2+it}}\bigg)^{-1},
\end{align*}
in the Steinhaus case. We would like to stress that this maneuver, to pass from an $L^2$-integral of the partial sums of $f$ to an $L^2$-integral of a product of independent random variables, is taken from Harper \cite[Proof of Proposition 2]{H3}. It differentiates from what was done in Lau--Tenenbaum--Wu \cite{LTW} in the fact that they arrived at a similar point, but kept working with the $L^2$-integral of partial sums of $f$, since such procedure would have not led them to a stronger result.

In conclusion, we may find
\begin{align}
\label{rewritingsmoothprob}
\mathbb{P}_1&\leq\mathbb{P}\bigg(\bigg\{\sup_{X_{\ell-1}<x_i\leq X_\ell}  \frac{\sqrt{\log\log x_i}}{\log x_i}\bigg(\frac{\log x_i}{\log X_{\ell-1}} \bigg)^{1/(\ell-1)^K}\\
&\times \int_{-\infty}^{+\infty}\bigg|\frac{\mathcal{S}_{x_i}(1/2+it)}{1/2+it}\bigg|^2dt>cT\bigg\}\bigg),\nonumber
\end{align}
for a certain $c>0$.

As it will be clear in a moment, the factors\footnote{Without them only the integral alone would have given rise to a submartingale sequence and a direct application of Doob's maximal inequality to handle the supremum in \eqref{rewritingsmoothprob} would have only led to an extremely large upper bound for $\mathbb{P}_1$.} $(\frac{\log x_i}{\log X_{\ell-1}})^{1/(\ell-1)^K}\geq 1$ have been introduced to make the sequence of random variables 
$$Y_{x_i}:=\frac{1}{\log x_i}\bigg(\frac{\log x_i}{\log X_{\ell-1}} \bigg)^{1/(\ell-1)^K}\int_{-\infty}^{+\infty}\bigg|\frac{\mathcal{S}_{x_i}(1/2+it)}{1/2+it}\bigg|^2dt$$ 
a submartingale sequence with respect to the filtration $\mathcal{F}_{i}:=\sigma(\{f(p):\ p\leq x_i\}).$ In fact, each $Y_{x_i}$ is certainly $\mathcal{F}_i$-measurable and $L^1$-bounded, since
\begin{align*}
\mathbb{E}[|\mathcal{S}_{x_i}(1/2+it)|^2]\ll \log x_i,
\end{align*}
by Lemma \ref{Eulerprodlem} and Mertens' theorem. Finally, we clearly have
\begin{align*}
\mathbb{E}[Y_{x_i}|\mathcal{F}_{i-1}]&=\frac{1}{\log x_{i-1}}\bigg(\frac{\log x_{i-1}}{\log X_{\ell-1}} \bigg)^{1/(\ell-1)^K}\frac{\log x_{i-1}}{\log x_i}\bigg(\frac{\log x_i}{\log x_{i-1}} \bigg)^{1/(\ell-1)^K}\\
&\times \int_{-\infty}^{+\infty} \frac{|\mathcal{S}_{x_{i-1}}(1/2+it)|^2}{|1/2+it|^2}\mathbb{E}\bigg[\frac{|\mathcal{S}_{x_i}(1/2+it)|^2}{|\mathcal{S}_{x_{i-1}}(1/2+it)|^2}\bigg]dt.
\end{align*}
By Lemma \ref{Eulerprodlem}, the expectation inside the integral equals
\begin{align*}
\exp\bigg(\sum_{x_{i-1}<p\leq x_i} \frac{1}{p}+O\bigg( \frac{1}{x_{i-1}}\bigg)\bigg)=\exp\bigg( \frac{\eps}{i}+O\bigg(\frac{1}{i^2}\bigg)\bigg),
\end{align*}
by a strong form of Mertens' theorem (with error term given by the Prime Number Theorem), if $i$ is sufficiently large with respect to $\eps$. On the other hand,
\begin{align*}
\frac{\log x_{i-1}}{\log x_i}=\bigg( 1-\frac{1}{i}\bigg)^{\eps}=\exp\bigg( -\frac{\eps}{i}+O\bigg(\frac{1}{i^2}\bigg)\bigg)
\end{align*}
and
\begin{align*}
\bigg( \frac{\log x_{i}}{\log x_{i-1}}\bigg)^{1/(\ell-1)^K}&\geq \bigg( \frac{\log x_{i}}{\log x_{i-1}}\bigg)^{\log 2/\log\log x_i}\\
&=\exp\bigg(\frac{\log 2}{i\log i}+O_\eps\bigg(\frac{1}{i^2\log i}\bigg) \bigg).
\end{align*}
We deduce that
\begin{align*}
\mathbb{E}[Y_{x_i}|\mathcal{F}_{i-1}]&\geq \frac{1}{\log x_{i-1}}\bigg(\frac{\log x_{i-1}}{\log X_{\ell-1}} \bigg)^{1/(\ell-1)^K}\int_{-\infty}^{+\infty} \frac{|\mathcal{S}_{x_{i-1}}(1/2+it)|^2}{|1/2+it|^2}dt\\
&=Y_{x_{i-1}},
\end{align*}
if $i$ is sufficiently large with respect to $\eps$. 
\subsection{Conditioning and low moments estimates}
We can now see \eqref{rewritingsmoothprob} as the probability that the supremum of a normalized submartingale sequence is large. This is the field where Doob's maximal inequality operates. However, as preannounced before, an immediate application of Lemma \ref{Doobmaxineq} turns out to be inefficient and to improve it we first need to introduce a conditioning on the values of $f$ at the small primes.
More specifically, we will condition on the following event:
\begin{align}
\label{eventlowmoments}
\Sigma_\ell:=\bigg\{ \int_{-\infty}^{+\infty} \frac{|\mathcal{S}_{X_{\ell-1}}(1/2+it)|^2}{|1/2+it|^2}dt\leq \frac{\sqrt{T}2^{(\ell-1)^K}}{\sqrt{(\ell-1)^K}}\bigg\}.
\end{align}
First of all, we need to check that $\Sigma_\ell$ holds with a probability sufficiently close to $1$. In fact, this is the more delicate part of our argument because we need access to deep information about the distribution of the Euler product of a random multiplicative function. By Markov's inequality for the power $1/2$, we get
\begin{align*}
\mathbb{P}(\overline{\Sigma_\ell})&\leq \bigg(\frac{\sqrt{(\ell-1)^K}}{\sqrt{T}2^{(\ell-1)^K}}\bigg)^{1/2}\mathbb{E}\bigg[\bigg(\int_{-\infty}^{+\infty} \frac{|\mathcal{S}_{X_{\ell-1}}(1/2+it)|^2}{|1/2+it|^2}dt\bigg)^{1/2}\bigg]\\
&\ll\frac{1}{T^{1/4}},
\end{align*}
whenever $\ell$ is sufficiently large with respect to $\eps$, which is good enough for Lemma \ref{lemS1}. Here, we have used Harper's low moment result, which gives 
\begin{align*}
\mathbb{E}\bigg[\bigg(\int_{-\infty}^{+\infty} \frac{|\mathcal{S}_{X_{\ell-1}}(1/2+it)|^2}{|1/2+it|^2}dt\bigg)^{1/2}\bigg]&\ll \bigg(\frac{\log X_{\ell-1}}{\sqrt{\log\log X_{\ell-1}}}\bigg)^{1/2}\\
&\ll\frac{2^{\frac{(\ell-1)^K}{2}}}{(\ell-1)^{K/4}}.
\end{align*}
This follows from \cite[Key Proposition 1]{H3} and \cite[Key Proposition 2]{H3}, as it is done in \cite{H3} in the paragraph entitled: ``Proof of the upper bound in Theorem 1, assuming Key Propositions 1 and 2''.

Now, by conditioning on the event $\Sigma_\ell$, we get that \eqref{rewritingsmoothprob} is at most
\begin{align}
\label{splittinglowmomentsprob}
\mathbb{P}\bigg(\bigg\{\sup_{X_{\ell-1}<x_i\leq X_\ell}  Y_{x_i}> \frac{cT}{\sqrt{\ell^K}}\bigg\}\bigg| \Sigma_\ell\bigg)+\frac{1}{T^{1/4}}.
\end{align}
By Doob's maximal inequality, Lemma \ref{Doobmaxineq}, the probability in \eqref{splittinglowmomentsprob} is
\begin{align*}
&\ll\frac{\sqrt{\ell^K}}{T}\mathbb{E}[Y_{X_{\ell}}|\Sigma_\ell]\\
&\leq \frac{\sqrt{\ell^K}}{T\log X_\ell}\bigg(\frac{2^{\ell^K}}{2^{(\ell-1)^K}}\bigg)^{1/(\ell-1)^K} \mathbb{E}\bigg[ \int_{-\infty}^{+\infty} \frac{|\mathcal{S}_{X_{\ell}}(1/2+it)|^2}{|1/2+it|^2}dt\bigg| \Sigma_\ell\bigg].
\end{align*}
Here, by abuse of notation, we indicated with $X_\ell$ the largest $x_i\leq X_\ell$.

By standard properties of the conditional expectation, we can rewrite the above expectation as 
\begin{align*}
&=\mathbb{E}\bigg[\mathbb{E}\bigg[ \int_{-\infty}^{+\infty} \frac{|\mathcal{S}_{X_{\ell-1}}(1/2+it)|^2}{|1/2+it|^2}\frac{|\mathcal{S}_{X_{\ell}}(1/2+it)|^2}{|\mathcal{S}_{X_{\ell-1}}(1/2+it)|^2}dt\bigg| \mathcal{F}_{\ell}\bigg]\bigg| \Sigma_\ell\bigg]\\
&=\mathbb{E}\bigg[ \int_{-\infty}^{+\infty} \frac{|\mathcal{S}_{X_{\ell-1}}(1/2+it)|^2}{|1/2+it|^2}\mathbb{E}\bigg[ \frac{|\mathcal{S}_{X_{\ell}}(1/2+it)|^2}{|\mathcal{S}_{X_{\ell-1}}(1/2+it)|^2}\bigg| \mathcal{F}_{\ell}\bigg]dt\bigg| \Sigma_\ell\bigg],
\end{align*}
where $\mathcal{F}_{\ell}:=\sigma(\{f(p):\ p\leq X_{\ell-1}\}).$

By Lemma \ref{Eulerprodlem} and Mertens' theorem, we get
\begin{align*}
\mathbb{E}\bigg[ \frac{|\mathcal{S}_{X_{\ell}}(1/2+it)|^2}{|\mathcal{S}_{X_{\ell-1}}(1/2+it)|^2}\bigg| \mathcal{F}_{\ell}\bigg]\ll \frac{\log X_\ell}{\log X_{\ell-1}}, 
\end{align*}
which inserted back gives an overall bound for the probability in \eqref{splittinglowmomentsprob}
\begin{align*}
&\ll_\eps \frac{\sqrt{\ell^K}}{T2^{(\ell-1)^K}}\mathbb{E}\bigg[ \int_{-\infty}^{+\infty} \frac{|\mathcal{S}_{X_{\ell-1}}(1/2+it)|^2}{|1/2+it|^2}dt\bigg| \Sigma_\ell\bigg],
\end{align*}
since 
$$2^{(\ell^K-(\ell-1)^K)/(\ell-1)^K}\ll_\eps 1.$$
Finally, reminding of the definition \eqref{eventlowmoments} of the event $\Sigma_\ell$, the above expression is $\ll_\eps 1/\sqrt{T},$ since $\ell^K/(\ell-1)^K\ll_\eps 1$, which is good enough for Lemma \ref{lemS1}, since $T\geq 1$. This concludes the proof of Lemma \ref{lemS1}.
\section*{Acknowledgements}
The author would like to thank his supervisor Adam J. Harper for guiding him through the work that led to this paper. 


\begin{thebibliography}{9}
\bibitem{BA} J. Basquin. \emph{Sommes friables de fonctions multiplicatives al\'eatoires.} Acta Arithmetica, 152.3 (2012).
\bibitem{BO} A. Bonami. \emph{\'{E}tude des coefficients de Fourier des fonctions de $L_p(G)$}. Ann. Inst. Fourier, Grenoble 20, (2) 335--402 (1970).
\bibitem{E} P. Erd\H{o}s. \emph{Some applications of probability methods to number theory.} In Proc. of the $4$th Pannonian Symp. on Math. Stat., (Bad Tatzmannsdorf, Austria 1983). Mathematical statistics and
applications, vol. B, pp 1--18. Reidel, Dordrecht (1985).
\bibitem{G} A. Gut. \emph{Probability: a graduate course}. New York: Springer (2005).
\bibitem{H} G. Hal\'{a}sz. \emph{On random multiplicative functions.} In Hubert Delange Colloquium, (Orsay, 1982).
Publications Math\'{e}matiques d'Orsay, 83, pp 74--96. Univ. Paris XI, Orsay (1983).
\bibitem{H1} A. J. Harper. \emph{Bounds on the suprema of Gaussian processes, and omega results for the sum of a
random multiplicative function}. Ann. Appl. Probab., 23, no. 2, pp 584--616 (2013).
\bibitem{H4} A. J. Harper. \emph{Moments of random multiplicative functions, II: High moments.} Algebra and Number Theory, Vol. 13, No. 10, 2277--2321 (2019).
\bibitem{H3} A. J. Harper. \emph{Moments of random multiplicative functions, I: Low moments, better than squareroot
cancellation, and critical multiplicative chaos.} Forum of Mathematics, Pi, 8, e1, 95pp. (2020).
\bibitem{H2} A. J. Harper. \emph{Almost sure large fluctuations of random multiplicative functions}. Preprint \url{https://arxiv.org/abs/2012.15809}.
\bibitem{HNR} A. J. Harper, A. Nikeghbali, M. Radziwi\l\l. \emph{A note on Helson's conjecture on moments of random
multiplicative functions.} Analytic number theory, pp 145--169, Springer, Cham. (2015).
\bibitem{HO} W. Hoeffding. \emph{Probability inequalities for sums of bounded random variables.} J. Amer.
Statist. Assoc. 58, 13--30 (1963).
\bibitem{LTW} Y.-K. Lau, G. Tenenbaum, J. Wu. \emph{On mean values of random multiplicative functions.} Proceedings of the American Mathematical Society. Volume 141, Number 2, pp. 409--420 (2013). Also see \url{https://tenenb.perso.math.cnrs.fr/PPP/RMF.pdf} for some corrections to the published version.
\bibitem{MV} H.~Montgomery, R.~Vaughan. \emph{Multiplicative Number theory I: Classical theory}. Cambridge U. P. (2006).
\bibitem{SH} P. Shiu. \emph{A Brun--Titchmarsh theorem for multiplicative functions}. J. Reine Angew. Math. 313, 161--170 (1980).
\bibitem{SOUND} K. Soundararajan. \emph{Partial sums of the M\"obius function}. J. reine angew. Math. 631, 141--152 (2009).
\bibitem{T} G. Tenenbaum. \emph{Introduction to Analytic and Probabilistic Number Theory}. Third edition, Graduate Studies in Mathematics, 163, American Mathematical Society, Providence, RI (2015).
\bibitem{W} A. Wintner. \emph{Random factorizations and Riemann's hypothesis}. Duke Math. J., 11, pp 267--275 (1944).
\end{thebibliography}
\end{document}